\documentclass[12pt, a4]{amsart}
\usepackage{amscd, amsmath, amssymb}
\usepackage{fullpage}
\usepackage{graphicx}

\theoremstyle{plain}
\newtheorem{theorem}{Theorem}[section]

\theoremstyle{definition}
\newtheorem{defn}[theorem]{Definition}
\newtheorem{example}[theorem]{Example}

\newtheorem{corollary}[theorem]{Corollary}
\numberwithin{equation}{section}

\newcommand{\R}{\mathbb{R}}

\begin{document}

\title[A parameter-dependent cantilever-type BVP]{On the solvability of a parameter-dependent cantilever-type BVP} 

\date{}

\author[G. Infante]{Gennaro Infante}
\address{Gennaro Infante, Dipartimento di Matematica e Informatica, Universit\`{a} della
Calabria, 87036 Arcavacata di Rende, Cosenza, Italy}%
\email{gennaro.infante@unical.it}%

\begin{abstract} 
We discuss the solvability of a parameter dependent cantilever-type boundary value problem. We provide an existence and localization result for the positive solutions via a Birkhoff-Kellogg type theorem. We also obtain, under additional growth conditions, upper and lower bounds for the involved parameters. An example is presented in order to illustrate the theoretical results.
\end{abstract}



\maketitle

\section{Introduction}

Differential equations have been utilized to model the steady states of deflections of elastic beams; for example the fourth order ordinary differential equation
\begin{equation}\label{de}
u^{(4)}(t)=f(t,u(t)),\ t \in (0,1),
\end{equation}
subject to the homogeneous boundary conditions (BCs)
\begin{equation}\label{debc}
u(0)=u'(0)=u''(1)=u'''(1)=0,
\end{equation}
can be used as a model for the so-called cantilever bar.  
The boundary value problem (BVP)~\eqref{de}-\eqref{debc} describes a bar of length $1$ which is clamped on the
left end and is free to move at the right end, with vanishing
bending moment and shearing force, see for example~\cite{da-jh, li, yao-cant}.

Under a mechanical point of view, some interesting cases appear when the shearing force at the right side of the beam does not vanish (see for example~\cite{gipp-cantilever}):
\begin{itemize}
\item $u'''(1)+k_0=0$ models a force acting in 1,
\item $u'''(1)+k_1 u(1)=0$ describes a spring in 1,
\item $u'''(1)+g(u(1))=0$ models
a spring with a strongly nonlinear rigidity,
\item $u'''(1)+g(u(\eta))=0$ describes
a feedback mechanism, where the spring reacts to the displacement registered in a point $\eta$ of the beam,
\item $u'''(1)+g(u(\eta_1), u'(\eta_2), u''(\eta_3))=0$ describes the case where the spring reacts to the displacement registered in the point~$\eta_1$, the angular attitude registered in the point~$\eta_2$ and the bending moment in the point~$\eta_3$. 
\end{itemize}
Of course, a different configuration of the beam may lead to more complicated  BCs than the ones illustrated above. It is therefore not surprising that the case of  non-homogeneous BCs has received attention by researchers. By means of critical point theory, Cabada and Terzian~\cite{cate} and Bonanno,  Chinn\`i and Terzian~\cite{gbacst} and Yang, Chen and Yang~\cite{ycy-aml}
studied the parameter-dependent BVP
\begin{equation*}
\begin{cases}
u^{(4)}(t)=\lambda f(t,u(t)),\ t \in (0,1),\\ 
u(0)=u'(0)=u''(1)=0,\\ u'''(1)+\lambda g(u(1))=0,
\end{cases}
\end{equation*}
while the case of $\lambda=1$ has been investigated in an earlier paper by Ma~\cite{tofu}.

By classical fixed point index, Cianciaruso, Infante and Pietramala~\cite{genupa} studied the BVP
\begin{equation*}
\begin{cases}
u^{(4)}(t)= f(t,u(t)),\ t \in (0,1),\\ 
u(0)=u'(0)=u''(1)=0,\\ u'''(1)+ \hat{H}(u)=0,
\end{cases}
\end{equation*}
where $\hat{H}$ is a suitable functional (not necessarily linear) on $C[0,1]$. The functional approach for the BCs adopted in~\cite{genupa} fits within the interesting framework of nonlinear and nonlocal BCs; these are widely studied objects, we refer the reader to the reviews~\cite{Cabada1, Conti, rma, Picone, sotiris, Stik, Whyburn} and the manuscripts~\cite{Chris-bj, kttmna, jw-gi-jlms}.

Regarding the higher order dependence in the forcing term, the ODE
$$
u^{(4)}(t)=f(t,u(t),u'(t),u''(t),u'''(t)), \quad t\in [0,1]
$$
under the homogeneous BCs
$$
u(0)=u'(0)=u''(1)=u'''(1)=0
$$
has been studied by Li~\cite{li} via fixed point index.
The non-homogeneous case
$$u(0)=u'(0)=u''(1)=0, u'''(1)+g(u(1))=0 $$ 
has been studied, with the lower and upper solutions method, by Wei, Li and Li~\cite{mwylgl21}, while
the case
$$u(0)=u'(0)=\int_0^1 p(t)u(t)\, dt, u''(1)=u'''(1)=\int_0^1 q(t)u''(t)\, dt,$$ 
has been investigated by Khanfer and Bougoffa~\cite{aklb21}
via the Schauder fixed point theorem.

Here we study the solvability of the parameter-dependent BVP
\begin{equation}\label{CantPar}
\begin{cases}
u^{(4)}(t)=\lambda f(t,u(t),u'(t),u''(t),u'''(t)),\ t \in (0,1),\\ 
u(0)=u'(0)=u''(1)=0,\\ u'''(1)+\lambda H[u]=0,
\end{cases}
\end{equation}
where $f$ is a continuous function, $H$ is a suitable compact functional in the space $C^3[0,1]$ (this allows higher order dependence in the BCs) and $\lambda$ is a non-negative parameter. For the existence result we adapt an approach used by the author~\cite{gi-BK}, in the context of elliptic systems, that relies on a Birkhoff-Kellogg type theorem in cones due to Krasnosel'ski\u{i} and Lady\v{z}enski\u{\i}~\cite{Kra-Lady}. We also provide, under additional growth conditions, a localization result for the parameter $\lambda$. The results are new and complement the ones present in the papers~\cite{gbacst, cate, genupa, aklb21, li, tofu, mwylgl21,  ycy-aml}. We also complement the results in~\cite{gi-ptrsa21}, by obtaining additional qualitative properties (such as monotonicity and localization) of the solution. We illustrate the applicability of our theoretical results in an example.
\section{Existence and localization of the eigenvalues}
First of all we associate to  the BVP~\eqref{CantPar} a perturbed Hammerstein integral equation of the form
\begin{equation}\label{chareq}
u(t)=\lambda \Bigl( \gamma(t)H[u]+ \int_{0}^{1} k(t,s) f(s,u(s),u'(s),u''(s),u'''(s))\,ds\Bigr),
\end{equation}
where the Green's function $k$ and the function $\gamma$ need to be determined; this is done by considering two auxiliary BVPs, a procedure found to be particularly useful in the case of nonlinear BCs, see for example~\cite{gi-mc} and references therein.

Regarding $k$ it is known (see for example Lemma~2.1 and Lemma~2.2 of~\cite{li}) that for $h\in C[0,1]$ the unique solution of the linear BVP
$$
\begin{cases}
u^{(4)}(t)=h(t),\ t \in (0,1),\\ u(0)=u'(0)=u''(1)=u'''(1)=0,
\end{cases}
$$
is given by
\begin{equation*}
u(t)=\int_{0}^{1} k(t,s) h(s)\,ds,
\end{equation*}
where
$$
 k(t,s)=\begin{cases} \frac{1}{6}(3t^2s-t^3),\ &s\geq t,\\ \frac{1}{6}(3s^2t-s^3),\ &s\leq t.
\end{cases}
 $$
Note that the function $k$ has the following properties
$$
k(t,s), \frac{\partial k}{\partial t}(t,s) , \frac{\partial^2 k}{\partial t^2}(t,s) \geq 0\ \text{on}\ [0,1]\times [0,1],  
$$
and
$$
\frac{\partial^3 k}{\partial t^3}(t,s) \leq 0\ \text{on}\ [0,1]^2\setminus \{(t,s)|t=s)\}.
$$

Regarding the function $\gamma$, note that 
(see  for example \cite{gipp-cantilever})
\begin{equation*}
\gamma (t)=\frac{1}{6}(3t^2-t^3)
\end{equation*} is the unique solution of the BVP
\begin{equation*}
\gamma^{(4)}(t)=0,\ \gamma(0)=\gamma'(0)=\gamma''(1)=0,\; \gamma'''(1)+1=0.
\end{equation*}
By direct calculation, it can be observed that
$$\gamma(t), \gamma'(t), \gamma''(t), -\gamma'''(t)\geq 0\ \text{on}\ [0,1].$$

With the above ingredients at our disposal, we can work in the space $C^3[0,1]$ endowed with the norm
$$\|u\|_3:=\max_{j=0,\ldots, 3}\{\|u^{(j)}\|_\infty\},\ \text{where}\ \|w\|_\infty=\sup_{t\in [0,1]}|w(t)|.$$

\begin{defn} We say that $\lambda$ is an~\emph{eigenvalue} of the BVP~\eqref{CantPar} with a corresponding eigenfunction  
 $u\in C^3[0,1]$ with $\|u\|_3>0 $ if the pair $(u,\lambda)$ satisfies the perturbed Hammerstein integral equation~\eqref{chareq}.
\end{defn}

We make use of the following Birkhoff-Kellogg type theorem in order to seek the eigenfunctions of the BVP~\eqref{CantPar}.
We recall that a cone $\mathcal{K}$ of a real Banach space $(X,\| \, \|)$ is a closed set with $\mathcal{K}+\mathcal{K}\subset \mathcal{K}$, $\mu \mathcal{K}\subset \mathcal{K}$ for all $\mu\ge 0$ and $\mathcal{K}\cap(-\mathcal{K})=\{0\}$. 
\begin{theorem}[Theorem~2.3.6,~\cite{guolak}]\label{B-K}
Let $(X,\| \, \|)$ be a real Banach space, $U\subset X$ be an open bounded set with $0\in U$, 
$\mathcal{K}\subset X$ be a cone, $T:\mathcal{K}\cap \overline{U}\to \mathcal{K}$
be compact and suppose that 
$$\inf_{x\in \mathcal{K}\cap \partial U}\|Tx\|>0.$$
Then there exist $\lambda_{0}\in (0,+\infty)$ and $x_{0}\in \mathcal{K}\cap \partial U$
such that $x_{0}=\lambda_{0} Tx_{0}$.
\end{theorem}
We apply the Theorem~\ref{B-K} in the cone 
\begin{equation}\label{con-cant}
K:=\bigl\{u\in  C^3[0,1]:\ u(t),u'(t),u''(t), -u'''(t)\ge 0, \ \text{for every}\ t\in [0,1]\bigr\}.
\end{equation}
The cone~\eqref{con-cant} is a smaller cone than the one of positive functions used in~\cite{gi-ptrsa21}, but larger than the one used in~\cite{li}. We consider the sets
$$K_{\rho}:=\{u\in K: \|u\|_3<\rho\},\,
\overline{K}_{\rho}:=\{u\in K: \|u \|_3\leq \rho\},
$$ 
$$
\partial
K_{\rho}:=\{u\in K: \|u\|_3=\rho\},$$
where $\rho \in (0,+\infty)$.

The following Theorem provides an existence result for an eigenfunction possessing a fixed norm and a corresponding positive eigenvalue.
\begin{theorem}\label{eigen}Let $\rho \in (0,+\infty)$ and assume the following conditions hold. 
\begin{itemize}

\item[$(a)$] 
$f\in C(\Pi_{\rho}, \R)$ and there exist $\underline{\delta}_{\rho} \in C([0,1],\R_+)$ such that
\begin{equation*}
f(t,u,v,w,z)\ge \underline{\delta}_{\rho}(t),\ \text{for every}\ (t,u,v,w,z)\in \Pi_{\rho},
\end{equation*}
where $$\Pi_{\rho}:[0,1]\times [0,\rho]^3 \times [-\rho,0].$$ 

\item[$(b)$] 
$H: \overline{K}_{\rho} \to \R$ is continuous and bounded. Let $\underline{\eta}_{\rho}\in [0,+\infty)$ be such that 
\begin{equation*}
H[u]\geq \underline{\eta}_{\rho},\ \text{for every}\ u\in \partial K_{\rho}.
\end{equation*}
\item[(c)] 
The inequality 
\begin{equation}\label{condc}
\underline{\eta}_{\rho}+\int_{0}^{1}\underline{\delta}_{\rho} (s)\,ds>0 
\end{equation}
holds.
\end{itemize}
Then the BVP~\eqref{CantPar} has a positive eigenvalue $\lambda_\rho$ with an associated eigenfunction $u_{\rho}\in \partial K_{\rho}$.

\end{theorem}
\begin{proof}
Let $Fu(t):= \int_{0}^{1} k(t,s) f(s,u(s),u'(s),u''(s),u'''(s))\,ds$ and $\Gamma u(t):=\gamma(t)H[u]$. 
Note that, 
due to the assumptions above, the operator $T=F+\Gamma$ maps $\overline{K}_{\rho}$ into $K$ and is compact; the compactness of $F$ 
follows from a careful use of  the Arzel\`{a}-Ascoli theorem (see~\cite{Webb-Cpt}) and $\Gamma$ is a finite rank operator. 

Take $u\in \partial K_{\rho}$, then we have
\begin{multline}\label{lwest}
\| Tu\|_{3}\geq \| (Tu)'''\|_{\infty}\geq |(Tu)'''(0)|\\ =\bigl| -H[u] - \int_{0}^{1}   f(s, u(s) ,u'(s),u''(s), u'''(s))\,ds\bigr |\\
=H[u]+  \int_{0}^{1}   f(s, u(s) ,u'(s),u''(s), u'''(s))\,ds \geq \underline{\eta}_{\rho}+\int_{0}^{1}\underline{\delta}_{\rho} (s)\,ds.
\end{multline}

Note that the RHS of \eqref{lwest} does not depend on the particular $u$ chosen. Therefore we have
$$
\inf_{u\in \partial K_{\rho}}\| Tu\|_3 \geq \underline{\eta}_{\rho}+\int_{0}^{1}\underline{\delta}_{\rho} (s)\,ds>0,
$$
and the result follows by Theorem~\ref{B-K}.
\end{proof}
The following Corollary provides an existence result for the existence of uncountably many couples of eigenvalues--eigenfunctions.
\begin{corollary}
\label{CorEig}
In addition to the hypotheses of Theorem~\ref{eigen}, assume that $\rho$  can be chosen arbitrarily in $(0,+\infty)$. Then for every $\rho$ there exists a non-negative eigenfunction $u_{\rho}\in \partial K_{\rho}$ of the BVP~\eqref{CantPar} to which corresponds a $\lambda_{\rho} \in (0,+\infty)$.
\end{corollary}

The next result provides some upper and lower bounds on the eigenvalues.

\begin{theorem}\label{eigenloc}
In addition to the hypotheses of Theorem~\ref{eigen} assume the following conditions hold.

\begin{itemize}
\item[$(d)$] 
There exist $\overline{\delta}_{\rho} \in C([0,1],\R_+)$ such that
\begin{equation*}
f(t,u,v,w,z)\le \overline{\delta}_{\rho}(t),\ \text{for every}\ (t,u,v,w,z)\in \Pi_{\rho}.
\end{equation*}

\item[$(e)$] Let $\overline{\eta}_{\rho}\in [0,+\infty)$ be such that 
\begin{equation*}
H[u]\leq \overline{\eta}_{\rho},\ \text{for every}\ u\in \partial K_{\rho}.
\end{equation*}
\end{itemize}
Then $\lambda_\rho$ satisfies the following estimates
$$
\frac{\rho}{\bigl( \overline{\eta}_{\rho}+ \int_{0}^{1}  \overline{\delta}_{\rho}(s)\,ds\bigr)}\leq \lambda_\rho\leq \frac{\rho}{\bigl( \underline{\eta}_{\rho}+ \int_{0}^{1}  \underline{\delta}_{\rho}(s)\,ds\bigr)}.
$$
\end{theorem}

\begin{proof}
By Theorem~\ref{eigen} there exist $u_{\rho}\in \partial K_{\rho}$ and $\lambda_{\rho}$ such that
\begin{equation}\label{eigensol}
u_{\rho}(t)=\lambda_{\rho}\Bigl( \gamma(t) H[u_{\rho}]+ \int_{0}^{1} k(t,s) f(s,u_{\rho}(s),\ldots, u_{\rho}'''(s))\,ds\Bigr).
\end{equation}
By differentiating \eqref{eigensol} we obtain
\begin{align*}
u_{\rho}'(t)=&\lambda_{\rho}\Bigl( \gamma'(t)  H[u_{\rho}]+ \int_{0}^{1} \frac{\partial k}{\partial t}(t,s)  f(s,u_{\rho}(s),\ldots, u_{\rho}'''(s))  \,ds\Bigr),\\
u_{\rho}''(t)=&\lambda_{\rho}\Bigl( \gamma'' (t) H[u_{\rho}]+ \int_{0}^{1} \frac{\partial^2 k}{\partial t^2}(t,s)  f(s,u_{\rho}(s),\ldots, u_{\rho}'''(s))\,ds\Bigr),\\
u_{\rho}'''(t)=&\lambda_{\rho}\Bigl( -H[u_{\rho}]- \int_{t}^{1}   f(s,u_{\rho}(s),u_{\rho}'(s),u_{\rho}''(s), u_{\rho}'''(s))\,ds\Bigr),
\end{align*}
which implies
\begin{equation}\label{3norm}
\ \|u_{\rho}'''\|_\infty=\lambda_{\rho}\Bigl( H[u_{\rho}]+ \int_{0}^{1}  f(s,u_{\rho}(s),\ldots, u_{\rho}'''(s))\,ds\Bigr).
\end{equation}
Furthermore note that 
$$
0\leq k(t,s), \frac{\partial k}{\partial t}(t,s) , \frac{\partial^2 k}{\partial t^2}(t,s)\leq 1,\ \text{on}\ [0,1]\times [0,1],  
$$
and
$$0\leq \gamma(t), \gamma'(t), \gamma''(t)\leq 1\ \text{on}\ [0,1],$$
which yield
$$
\rho=\ \|u_{\rho}\|_3=\ \|u_{\rho}'''\|_\infty.
$$
From \eqref{3norm} and the estimates $(d)$ and $(e)$ we obtain
$$
\rho=\lambda_{\rho}\Bigl( H[u_{\rho}]+ \int_{0}^{1} f(s,u_{\rho}(s),\ldots, u_{\rho}'''(s))\,ds\Bigr)
\leq \lambda_{\rho} \bigl( \overline{\eta}_{\rho}+ \int_{0}^{1}  \overline{\delta}_{\rho}(s)\,ds\bigr),
$$
and 
$$
\rho=\lambda_{\rho}\Bigl( H[u_{\rho}]+ \int_{0}^{1}  f(s,u_{\rho}(s),\ldots, u_{\rho}'''(s))\,ds\Bigr)
 \geq \lambda_{\rho} \bigl( \underline{\eta}_{\rho}+ \int_{0}^{1}  \underline{\delta}_{\rho}(s)\,ds\bigr),
$$
which proves the result.
\end{proof}
We conclude with an example that illustrates the applicability of the previous theoretical results.
\begin{example}
Consider the BVP
\begin{equation} \label{example}
\begin{cases}
u^{(4)}(t)=\lambda te^{u(t)}(1+(u'''(t))^2),\ t \in (0,1),\\ 
u(0)=u'(0)=u''(1)=0,\\ u'''(1)+\lambda \Bigl(\frac{1}{1+(u(\frac{1}{2}))^2}+\int_0^1t^3u''(t)\, dt \Bigr)=0.
\end{cases}
\end{equation}
Fix $\rho\in (0,+\infty)$. Thus we may take 
$$\underline{\eta}_{\rho}(t)=\frac{1}{1+\rho^2},\, \overline{\eta}_{\rho}(t)=1+\frac{\rho}{4},\, \underline{\delta}_{\rho}(t)=t,\, \overline{\delta}_{\rho}(t)=te^{\rho}(1+\rho^2).$$ 

Thus we have
$$
\underline{\eta}_{\rho}+\int_{0}^{1}\underline{\delta}_{\rho} (s)\,ds=\frac{1}{1+\rho^2}+\int_0^1 s\, ds\geq \frac{1}{2},
$$
which implies that \eqref{condc} is satisfied for every $\rho \in (0,+\infty)$.

Thus we can apply Corollary~\ref{CorEig} and Theorem~\ref{eigenloc}, obtaining uncountably many pairs of positive eigenvalues and eigenfunctions $(u_{\rho}, \lambda_{\rho})$ for the BVP~\ref{example}, where 
$\|u_{\rho}\|_3=\|u_{\rho}'''\|_{\infty}=\rho$ and
$$
\frac{4 \rho}{2 {\mathrm e}^{\rho} \rho^{2}+2 {\mathrm e}^{\rho}+\rho +4}\leq \lambda_\rho\leq \frac{2 \rho  \left(\rho^{2}+1\right)}{\rho^{2}+3}.
$$
The Figure~\ref{Region} (produced with the program MAPLE) illustrates the region of localization of the $(u_{\rho}, \lambda_{\rho})$ pairs.
\begin{figure}[h]
\includegraphics[height=5.5cm,angle=0]{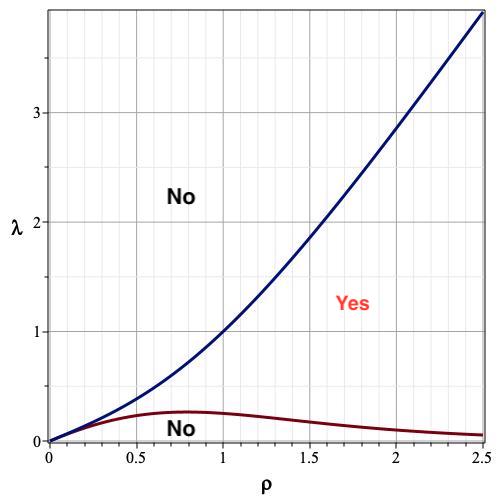}
\caption{Localization of $(u_{\rho}, \lambda_{\rho})$}\label{Region}
\end{figure}
\end{example}

\section*{Acknowledgement}
G. Infante was partially supported by G.N.A.M.P.A. - INdAM (Italy). This research has been accomplished within the UMI Group TAA “Approximation Theory and Applications”.


\begin{thebibliography}{99}

\bibitem{da-jh} D. R. Anderson and J. Hoffacker,
 Existence of solutions for a cantilever beam problem,
\textit{J. Math. Anal. Appl.}, \textbf{323} (2006), 958--973.

\bibitem{gbacst} G. Bonanno, A. Chinn\`i and S. Tersian,
Existence results for a two point boundary value problem involving a fourth-order equation,
\textit{Electron. J. Qual .Theory Differ. Equ.}, \textbf{33} (2015), 9pp.

\bibitem{Cabada1} A. Cabada, An overview of the lower and upper solutions method with nonlinear boundary value conditions, \textit{Bound. Value Probl.} (2011), Art. ID 893753, 18 pp.

\bibitem{cate} A. Cabada and S. Tersian,
Multiplicity of solutions of a two point boundary value problem for a fourth-order equation,
\textit{Appl. Math. Comput.}, \textbf{219} (2013), 5261--5267.

\bibitem{genupa}
F. Cianciaruso, G. Infante and P. Pietramala, Solutions of perturbed Hammerstein integral equations with applications, \textit{Nonlinear Anal. Real World Appl.}, \textbf{33} (2017), 317--347.

\bibitem{Conti} R. Conti,
Recent trends in the theory of boundary value problems for ordinary differential equations,
\textit{Boll. Un. Mat. Ital.}, \textbf{22} (1967), 135--178.

\bibitem{Chris-bj} C. S. Goodrich,  Pointwise conditions for perturbed Hammerstein integral equations with monotone nonlinear, nonlocal elements, \textit{Banach J. Math. Anal.}, \textbf{14} (2020), 290--312.

\bibitem{guolak} D. Guo and V. Lakshmikantham,
\textit{Nonlinear problems in abstract cones}, Academic Press, Boston,
1988. 

\bibitem{gi-mc}
G. Infante, A short course on positive solutions of systems of ODEs via fixed point index, in \textit{Lecture Notes
in Nonlinear Analysis (LNNA)}, \textbf{16} (2017), 93--140.

\bibitem{gi-ptrsa21}
G. Infante,
Positive solutions of systems of perturbed Hammerstein integral equations with arbitrary order dependence,
\textit{Philos. Trans. Roy. Soc. A}, \textbf{379} (2021), no. 2191, Paper No. 20190376, 10 pp.

\bibitem{gi-BK}
G. Infante, Eigenvalues of elliptic functional differential systems via a Birkhoff-Kellogg type theorem, \textit{Mathematics}, \textbf{9} (2021), n.~4.

\bibitem{gipp-cantilever} G. Infante and P. Pietramala,
A cantilever equation with nonlinear boundary conditions,
\textit{Electron. J. Qual. Theory Differ. Equ., Spec. Ed. I},
\textbf{15} (2009), 1--14.

\bibitem{kttmna} G. L. Karakostas and P. Ch. Tsamatos, Existence of multiple
positive solutions for a nonlocal boundary value problem,
\textit{Topol. Methods Nonlinear Anal.}, \textbf{19} (2002),
109--121.

\bibitem{aklb21}
A. Khanfer and L. Bougoffa,
A cantilever beam problem with small deflections and perturbed boundary data,
\textit{J. Funct. Spaces}, \textbf{2021}, Article ID 9081623, 9 p. (2021).

\bibitem{Kra-Lady}
M. A. Krasnosel'ski\u{i} and L. A. Lady\v{z}enski\u{\i}, The structure of the spectrum of positive nonhomogeneous operators, \textit{Trudy Moskov. Mat. Ob\v{s}\v{c}}, \textbf{3} (1954), 321--346.

\bibitem{li}
Y. Li, Existence of positive solutions for the cantilever beam
equations with fully nonlinear terms, \textit{Nonlinear Anal. Real
World Appl.}, \textbf{27} (2016), 221--237.

\bibitem{tofu}
T. F. Ma, Positive solutions for a beam equation on a nonlinear elastic foundation, \textit{Math. Comput. Modelling}, \textbf{39} (2004), 1195--1201.

\bibitem{rma}
R. Ma, A survey on nonlocal boundary value problems, \textit{Appl.
Math. E-Notes}, \textbf{7} (2007), 257--279.

\bibitem{sotiris}
S. K. Ntouyas, Nonlocal initial and boundary value problems: a
survey, \textit{Handbook of differential equations: ordinary
differential equations. Vol. II}, Elsevier B. V., Amsterdam,
(2005), 461--557.

\bibitem{Picone} M. Picone, Su un problema al contorno nelle equazioni differenziali lineari ordinarie del secondo ordine,
\textit{Ann. Scuola Norm. Sup. Pisa Cl. Sci.}, \textbf{10} (1908), 1--95.

\bibitem{Stik} A. \v{S}tikonas, A survey on stationary problems, Green's functions and spectrum of Sturm-Liouville problem
with nonlocal boundary conditions, \textit{Nonlinear Anal. Model.
Control}, \textbf{19} (2014), 301--334.

\bibitem{Webb-Cpt}
J. R. L. Webb, Compactness of nonlinear integral operators with discontinuous and with singular kernels, \textit{J. Math. Anal. Appl.}, \textbf{509} (2022),  Paper No. 126000, 17 pp.

\bibitem{jw-gi-jlms} J. R. L. Webb and G. Infante,
Positive solutions of nonlocal boundary value problems: a unified
approach, \textit{J. London Math. Soc.}, \textbf{74} (2006),
673--693.

\bibitem{mwylgl21}
M. Wei, Y. Li and G. Li, Lower and upper solutions method to the fully elastic cantilever beam equation with support,
\textit{Adv. Difference Equ.}, \textbf{2021}, 301 (2021).


\bibitem{Whyburn}
W. M. Whyburn, Differential equations with general boundary
conditions, \textit{Bull. Amer. Math. Soc.}, \textbf{48} (1942),
692--704.

\bibitem{ycy-aml}
L. Yang, H. Chen and X. Yang, 
The multiplicity of solutions for fourth-order equations generated from a boundary condition,
\textit{Appl. Math. Lett.}, \textbf{24} (2011), 1599--1603.

\bibitem{yao-cant}
Q. Yao, Monotonically iterative method of nonlinear cantilever beam equations,
\textit{Appl. Math. Comput.}, \textbf{205} (2008), 432--437.

\end{thebibliography}
\end{document}